\documentclass[11pt]{article}
\usepackage[letterpaper,top=2cm,bottom=2cm,left=2cm,right=2cm,marginparwidth=1.75cm]{geometry}
\usepackage{amsmath,amsthm,amssymb,amscd,mathdots,stmaryrd,enumerate,shuffle,ytableau,varwidth,faktor,xfrac,xcolor,calc,titlepic,setspace,enumitem,graphicx,tikz}
\usetikzlibrary{positioning}
\usepackage{mathdots}
\usepackage{MnSymbol}
\usepackage[utf8]{inputenc}
\usepackage[nottoc]{tocbibind}
\usepackage[
backend=biber,
style=alphabetic,
sorting=ynt
]{biblatex}
\title{Bibliography management: \texttt{biblatex} package}
\author{Share\LaTeX}
\date{ }

\graphicspath{ {./images/} }

\newtheorem{theorem}{Theorem}[section]
\newtheorem{proposition}[theorem]{Proposition}

\newtheorem{lemma}[theorem]{Lemma}

\numberwithin{equation}{section}
\theoremstyle{definition}

\newtheorem{remark}[theorem]{Remark}

\newcommand{\R}{\operatorname{\mathbb{R}}}

\newcommand{\C}{\operatorname{\mathbb{C}}}
\newcommand{\Z}{\operatorname{\mathbb{Z}}}
\newcommand{\re}{\operatorname{Re}}
\newcommand{\im}{\operatorname{Im}}

\begin{document}
\begin{center}
\vspace{5em}
{\Large\textsc{Bounding the integral of the argument of the Riemann Zeta function}}\\
\vspace{1em}
{\large\textsc{Victor Simon Gerard Amberger}}
\end{center}
\begin{abstract}
	\textsc{Abstract.} This article improves the estimate of $|S_1(t_2)-S_1(t_1)|$, which is the definite integral of the argument of the Riemann zeta-function between $t_1$ and $t_2$. Estimates of this quantity are needed to apply Turing's method to compute the exact number of zeta zeros up to a given height.
\end{abstract}

{\textsc{Keywords:}} Turing's Method, Argument of Zeta function, Zeros of Zeta function, Integral of Argument of Zeta function

MSC Classification: 11M26


\section{Introduction}\label{sec1}
The distribution of the non-trivial zeros of the Riemann zeta function is of great importance to many estimates related to the Riemann zeta function. It is conjectured by the Riemann Hypothesis that all zeros lie on the critical line and the conjecture has been proven up to a height of $3\cdot10^{12}$ \cite{Platt-Trudgian}.
One way to count the number of zeros of the Riemann zeta function inside critical strip is to first count zeros on the critical line and then verify that no zeros have been skipped. One method to find a lower bound for the number of zeros on the critical line is to search for sign changes in the Riemann-Siegel $Z$ function at Gram points. The Riemann-Siegel $Z$ function is given by
\begin{equation}\label{Riemann Siegal Z}
    Z(t)=e^{i\theta(t)}\zeta\left(\frac{1}{2}+it\right),
\end{equation}
where
\begin{equation}
    \theta(t)=\im\log\Gamma\left(\frac{1}{4}+\frac{it}{2}\right)-\frac{t\log\pi}{2}
\end{equation}
(see page 89 of \cite{titchmarsh}), and the $n$-th Gram point is defined by $\theta(g_n)=\pi n$ \cite[pp. 125-126]{edwards}. Because \eqref{Riemann Siegal Z} defines a real function we can deduce that $\zeta\left(\frac{1}{2}+ig_n\right)$ is real for all Gram points. More specifically,
\begin{equation*}
    \zeta\left(\frac{1}{2}+ig_n\right)=(-1)^nZ(g_n),
\end{equation*}
thus if we find consecutive Gram points where the sign of $\re\zeta\left(\frac{1}{2}+ig_n\right)$ does not change, then we located an interval $[g_n,g_{n+1})$ where $Z(t)$ has an odd number of zero crossings. Gram's law observes that $\zeta\left(\frac{1}{2}+ig_n\right)$ is usually positive and that each Gram interval usually contains exactly one zeta zero \cite{hutchinson}. This observation is closely linked to the behavior of the argument of the Riemann zeta function, given by
\begin{equation*}
    \label{Arg zeta}
    S(t):=\frac{1}{\pi}\Delta_C\arg\zeta(s)
\end{equation*}
when $t$ is not the imaginary part of a zeta zero. Otherwise, define $S(t)$ by the principal value. Here, $\Delta_C\arg$ denotes the change of the argument along the path $C$, which consists of the lines that connect $2$ and $2+it$ as well as $2+it$ and $\frac{1}{2}+it$. The argument of the zeta function is closely related to the number of zeros in the critical strip, a connection that is made explicit by the following relation \cite[eq. 3.1]{Brent},
\begin{equation}\label{N(t), S(t), connection}
    N(t)=1+\frac{1}{\pi}\theta(t)+S(t).
\end{equation}
where $N(t)$ counts the number of zeros inside the critical strip up to a given height, formally defined by
\begin{equation*}
    \begin{split}
        N(t)=\frac{1}{2}\#\left\{0<\beta<1,\ 0<\gamma<t\ \mid\ \zeta(\beta+i\gamma)=0\right\}\\
        +\frac{1}{2}\#\left\{0<\beta<1,\ 0<\gamma\leq t\ \mid\ \zeta(\beta+i\gamma)=0\right\}
    \end{split}
\end{equation*}
counted with multiplicity. Observe that \eqref{N(t), S(t), connection} implies that Gram's law is equivalent to $|S(t)|<1$. However, this assertion has to fail eventually due to a result of Selberg \cite{Selberg} which states that
\begin{equation*}
    S(t)=\Omega_{\pm}\left(\frac{(\log t)^\frac{1}{3}}{(\log\log t)^\frac{7}{3}}\right).
\end{equation*}
Indeed, the first Gram point at which Gram's law fails and where $\zeta\left(\frac{1}{2}+ig_n\right)<0$ is $n=126$ \cite{hutchinson}. A modified version of Gram's law is Rosser's rule which considers Gram blocks instead of Gram intervals. A Gram block of length $l$ is an interval $[g_n,g_{n+l})$ where $\zeta\left(\frac{1}{2}+ig_{n+j}\right)>0$ for $j=n,n+l$ and $\zeta\left(\frac{1}{2}+ig_{n+j}\right)\leq0$ for $j\in(n,n+j)$ \cite[\S5]{Lehman}. This may be understood as a smoothed version of the Gram interval. And Rosser's rule states that each Gram block of length $l$ usually contains $l$ zeta zeros. This holds true if and only if $|S(t)|<2$ for $t\in[g_n,g_{n+l})$.
So, this method is easier to apply than Gram's law, but ultimately it is doomed to fail all the same. Another method for obtaining an upper bound on $N(t)$ and test the list of zeros obtained from counting zero crossings of $Z(t)$ for completeness is Turing's method. This relies not on bounding $S(t)$ but the average $S_1(t)$ instead. The idea is to exploit the fact that $S(t)$ is a highly oscillating function that satisfies
\begin{equation*}
    S_1(t):=\int_0^tS(u)du=O(\log t),
\end{equation*}
due to Littlewood \cite[pp. 221-222]{titchmarsh}. Turing uses an explicit version of this bound to show the following result. We refer to the work of Lehman who fixed some mistakes in Turing's argument. Based on this work, Brent proved the following theorem.
\begin{theorem}[Brent-Lehman{\cite[Theorem 3.2]{Brent}}]
    \label{brent}
    Let $a$ and $b$ be positive real constants that satisfy
    \begin{equation*}
        |S_1(t_2)-S_1(t_1)|\leq a+b\log t_2
    \end{equation*}
    for all $t_2>t_1\geq t_0$. Furthermore, let  $K$ consecutive Gram Blocks with union $[g_n,g_p)$ satisfy Rosser's rule, such that $g_n\geq t_0$ and
    \begin{equation*}
        K\geq\frac{b}{6\pi}\log^2 g_p+\frac{a-b\log2\pi}{6\pi}\log g_p,
    \end{equation*}
    then
    \begin{equation*}
        N(g_n)\leq n+1\ \text{ and }\ N(g_p)\geq p+1.
    \end{equation*}
\end{theorem}
Regrettably, however, there is a minor mistake in Lehman's argument and the due to that the lower bound on $K$ is off by a factor of $2$. This error is also present in Brent's and Trudgian's estimates for $K$.\\
We will present a corrected version of this theorem in section \ref{sec3}. Section \ref{section estimate S_1} contains an improved estimate of $|S_1(t_2)-S_1(t_1)|$ and section \ref{section Tables} contains tables with additional constants for this estimate, as well as comparisons to previous estimates. The main improvement in this paper stems from the use of contributions of the real part of the second logarithmic derivative of the zeta function to bound the integral of its logarithm, rather than only using contributions of the logarithmic derivative. In addition, further, smaller improvements result from the use of a better estimate of the logarithmic derivatives of the zeta function in the right half plane with $\re(s)>1$ (see remark \ref{tatort time}) and the use of newer, improved bounds for the zeta function on both the $1/2$- and the $1$-line, due to Hiary, Patel, and Yang \cite{Hiary-Patel-Yang}, and Qingyi, and Teo \cite{Qingyi-Teo} respectively.
\section{Results}\label{sec2}
The main result of this paper is an improved estimate of $S_1(t)$.
\begin{theorem}\label{S_1 estimates theorem}
    For $t_2>t_1\geq 653$
    \begin{equation*}
        |S_1(t_2)-S_1(t_1)|\leq a+b\log\log t_2+c\log t_2,
    \end{equation*}
    with $(a,b,c)=(1.680,0.186,0.0314)$.
\end{theorem}
\begin{remark}
    The choice of constants $a,b,c$ is made to be optimal around the height of $10^{12}$. Further admissible choices for $(a,b,c)$ are given in table \eqref{table: e}.
\end{remark}
Estimates of this kind can be used to prove results about the distribution of non-trivial zeta zeros. One application is Turing's method to verify that a given list of non-trivial zeta zeros is complete. Lagarias \cite{lagarias} uses results of this kind to give explicit bounds on the horizontal gap between consecutive zeros. Although these bounds do not picture the asymptotic spacing between consecutive zeta zeros, they give a tighter bound on the gaps between zeros at relatively low heights.
In addition to these estimates, we correct a mistake in Theorem \ref{brent} and show a corrected version that works with $S_1(t)$ estimates of the form $a+b\log\log t+c\log t$.
\begin{theorem}\label{theorem abc loglog t}
    Let $a$, $b$ and $c$ be positive real constants that satisfy
    \begin{equation*}
        |S_1(t_2)-S_1(t_1)|\leq a+b\log\log t_2 +c\log t_2
    \end{equation*}
    for all $t_2>t_1\geq t_0$. Furthermore, let  $K$ consecutive Gram Blocks with union $[g_n,g_p)$ satisfy Rosser's rule, such that $g_n\geq t_0$ and
    \begin{equation*}
        K>\frac{1}{3\pi}\left(\log\frac{g_p}{2\pi}+\frac{0.048\pi^2}{g_p^2}\right)\left(a+b\log\log g_p+c\log g_p\right),
    \end{equation*}
    then
    \begin{equation*}
        N(g_n)\leq n+1\ \text{ and }\ N(g_p)\geq p+1.
    \end{equation*}
\end{theorem}
\begin{remark}
    Substituting $b=0$ and neglecting the very small $O(g_p^{-2})$ error term yields the corrected version of Theorem \ref{brent}.
\end{remark}
\section{Proof of Theorem \ref{theorem abc loglog t}}\label{sec3}
First, we prove Theorem \ref{theorem abc loglog t}. The proof is analogous to Brent's proof of Theorem \ref{brent} and relies on the following lemma due to Lehman. 
\begin{lemma}[{\cite[Lemma 16]{Lehman}}]
    \label{lemma queer bitch}
    Let $t^*\in(0.988,1.012)$ be the local minimum of $\kappa(t)$ for $0<t<\infty$, where
    \begin{equation*}
        \kappa(t)=-\frac{1}{2}t\log\pi +\frac{1}{4\pi i}\log\frac{\Gamma\left(\frac{1}{4}+i\pi t\right)}{\Gamma\left(\frac{1}{4}-i\pi t\right)}=\frac{1}{2\pi}\theta(2\pi t)
    \end{equation*}
    and $\kappa(t^*)=-\frac{9}{16}+\Theta(0.005)$. Then
    \begin{equation*}
        \left|\int_y^xS(2\pi t)\kappa'( t)d\tau\right|\leq\frac{1}{4\pi}\left(\log x+\frac{0.012}{x^2}\right)\max_{y\leq t\leq x}|S_1(2\pi x)-S_1(2\pi t)|
    \end{equation*}
    for $t^*<y<x$.
    Equivalently,
    \begin{equation*}
        \left|\int_{2\pi y}^{2\pi x}S(t)\theta'(t)dt\right|\leq\frac{1}{2}\left(\log x+\frac{0.012}{x^2}\right)\max_{y\leq t\leq x}|S_1(2\pi x)-S_1(2\pi t)|.
    \end{equation*}
\end{lemma}
We may use this lemma to prove Theorem \ref{theorem abc loglog t}.
\begin{proof}[Proof of Theorem \ref{theorem abc loglog t}]
    The function $N(t)$ increases by an odd number in a Gram interval $[g_n,g_{n+1})$ if both $\zeta\left(\frac{1}{2}+ig_n\right)$ and $\zeta\left(\frac{1}{2}+ig_{n+1}\right)$ have the same sign. If they have different signs, then $N(t)$ increases by an even number or stays the same. Because $\zeta\left(\frac{1}{2}\right)<0$ and by assumption $\zeta\left(\frac{1}{2}+ig_n\right)>0$ we obtain that $N(g_n)=n+1+2q$ for some $q\in\Z$. Consider a Gram block $[g_m,g_{m+k})$ such that $n\leq m<m+k\leq p$. This contains at least $k$ zeros of $Z(t)$, hence
    \begin{equation*}
        S(t)\geq m+2q-\frac{1}{\pi}\theta(t)
    \end{equation*}
    for $g_m<t<g_{m+1}$ and if $k>1$ then $Z(t)$ changes sign in each interval $(g_{m+j},g_{m+j+1})$ for all $j\in[1,2,\dots,k-1]$. Hence,
    \begin{equation*}
        S(t)\geq m+j-1+2q-\frac{1}{\pi}\theta(t)
    \end{equation*}
    for $g_{m+j}<t<g_{m+j+1}$. Now, assume that $q\geq1$, then
    \begin{equation*}
        \begin{split}
            \frac{1}{\pi}\int_{g_m}^{g_{m+k}}S(t)\theta'(t)dt\geq&\frac{1}{\pi}\int_{g_m}^{g_{m+1}}\left(m+2-\frac{1}{\pi}\theta(t)\right)\theta'(t)dt\\
            &+\frac{1}{\pi}\sum_{j=1}^{k-1}\int_{g_{m+j}}^{g_{m+j+1}}\left(m+1+j-\frac{1}{\pi}\theta(t)\right)\theta'(t)dt\\
            =&\int_{\frac{1}{\pi}\theta(g_m)}^{\frac{1}{\pi}\theta(g_{m+1})}(m+2-x)dx+\sum_{j=1}^{k-1}\int_{\frac{1}{\pi}\theta(g_{m+j})}^{\frac{1}{\pi}\theta(g_{m+j+1})}(m+1+j-x)dx\\
            =&\int_m^{m+1}(m+2-x)dx+\sum_{j=1}^{k-1}\int_{m+j}^{m+j+1}(m+1+j-x)dx\\
            =&1+\frac{1}{2}k.
        \end{split}
    \end{equation*}
    Combining this estimate for $K$ consecutive Gram blocks yields
    \begin{equation*}
        \int_{g_n}^{g_p}S(t)\theta'(t)dt\geq\pi\left(K+\frac{1}{2}(p-n)\right).
    \end{equation*}
    By assumption $p-n\geq K$, therefore we obtain
    \begin{equation*}
        \int_{g_n}^{g_p}S(t)\theta'(t)dt\geq\frac{3\pi}{2}K
    \end{equation*}
    and by our assumption on $K$ this contradicts lemma \ref{lemma queer bitch}. Thus, we may deduce that $q\leq0$ and $N(g_n)\leq n+1$. Similarly, one may deduce that $N(g_p)=p+1+2q'$ for some $q'\in\Z$. Consider again a Gram block $[g_m,g_{m+k})$ such that $n\leq m<m+k\leq p$. By assumption, this block contains $k$ zeros of $Z(t)$, and hence
    \begin{equation*}
        S(t)\leq p+2q'-\frac{1}{\pi}\theta(t)
    \end{equation*}
    for $g_{m+k-1}<t<g_{m+k}$. Furthermore, we find that if $k>1$ then $Z(t)$ changes sign in each interval $(g_{m+j},g_{m+j+1})$ for all $j\in[0,1,\dots,k-2]$. Hence,
    \begin{equation*}
        S(t)\leq m+j+2+2q'-\frac{1}{\pi}\theta(t)
    \end{equation*}
    for $g_{m+j}<t<g_{m+j+1}$. Now assume that $q'<0$, then
    \begin{equation*}
        \int_{g_m}^{g_p}S(t)\theta'(t)dt\leq-\pi\left(K+\frac{1}{2}(p-n)\right).
    \end{equation*}
    This again yields a contradiction to lemma \ref{lemma queer bitch} and we may conclude that $q'\geq0$\\ and $N(g_p)\geq p+1$.
\end{proof}
\section{Estimates for $S_1(t)$}\label{section estimate S_1}
In this section, we show the improved bounds for $S_1(t)$ presented in Theorem \ref{S_1 estimates theorem}. The proof relies on the following lemmas.
\begin{lemma}\label{switch from vertical to horizontal lines}
    If $t_2>t_1>0$, then
    \begin{equation*}
        \pi\int_{t_1}^{t_2}S(t)dt=\re\int_{\frac{1}{2}+it_2}^{\infty+it_2}\log\zeta(s)ds-\re\int_{\frac{1}{2}+it_1}^{\infty+it_1}\log\zeta(s)ds
    \end{equation*}
\end{lemma}
\begin{proof}
    This is Lemma 2.4 in \cite{Trudgian1} or alternatively Lemma 1 in \cite{Lehman2} for a more detailed proof.
\end{proof}
So in order to estimate $S_1(t)$ we need both upper and lower bounds for 
\begin{equation}\label{integral re log zeta}
    \re\int_{\frac{1}{2}+it}^{\infty+it}\log\zeta(s)ds.
\end{equation}
We will first focus on the upper bound. For this we rely on the following version of the Phragmén-Lindelöf principle due to Fiori in conjunction with the best current bounds for $|\zeta(s)|$ on both the $\frac{1}{2}$- and the $1$-line.
\begin{proposition}
    For $t\geq 3$, we have
    \begin{equation}\label{hiary,patel,yang}
        \left|\zeta\left(\frac{1}{2}+it\right)\right|\leq0.618t^\frac{1}{6}\log t.
    \end{equation}
    For $t\geq653$, we have
    \begin{equation}\label{qingyi,teo}
        |\zeta(1+it)|\leq0.548\log t.
    \end{equation}
\end{proposition}
\begin{proof}
    Equation \eqref{hiary,patel,yang} is shown by Hiary, Patel and Yang in \cite{Hiary-Patel-Yang} and equation \eqref{qingyi,teo} is shown by Qingyi and Teo in \cite{Qingyi-Teo}.
\end{proof}
With these bounds we can prove the following bounds on $|\zeta(s)|$ with real part $\sigma\geq\frac{1}{2}$.
\begin{lemma}\label{fiori}
    For $t\geq t_0\geq653$, $1\geq\delta>0$, and $\sigma\in\left[\frac{1}{2},1+\delta\right]$, we have
    \begin{equation*}
        \begin{split}
            |\zeta(\sigma+it)|\leq& C_1(653)\zeta(1+\delta)^\frac{\sigma-1}{\delta}\left(0.548\log t\right)^\frac{1+\delta-\sigma}{\delta}\ \text{ for }1\leq\sigma\leq1+\delta\\
            |\zeta(\sigma+it)|\leq&C_1(653)\left(0.548\right)^{2\sigma-1}\left(0.618t^\frac{1}{6}\right)^{2-2\sigma}\log t\ \text{ for }\frac{1}{2}\leq\sigma\leq1,
        \end{split}
    \end{equation*}
    where $C_i(t)$ is decreasing in $t$ with
    \begin{equation*}
        C_1(653)\leq1+6.23\cdot10^{-6},\ \text{ and }C_2(653)\leq1+6.43\cdot10^{-6}.
    \end{equation*}
\end{lemma}
\begin{proof}
    Apply Theorem 7 from \cite{fiori} with 
    \begin{equation*}
        G(s)=\frac{\log(e+s)+\log(e+2-s)}{2}
    \end{equation*}
    and $\alpha=1,\beta=0$. We check that $|G(s)|\geq1$ and $|G(1+it)|\geq\log t$ for $t\geq653$. Furthermore, $|G(\sigma+it)|$ is increasing in $\sigma$ for $t\geq653$. So we see that all the requirements of the theorem are satisfied and together with the bound given in \eqref{qingyi,teo} we obtain that
    \begin{equation*}
        \left|\frac{s-1}{s}\zeta(s)\right|\leq\zeta(1+\delta)^\frac{\sigma-1}{\delta}0.548^\frac{1+\delta-\sigma}{\delta}G(s)^\frac{1+\delta-\sigma}{\delta}.
    \end{equation*}
    Notice that
    \begin{equation}\label{weihnachts-quiz}
        \left|\frac{\sigma+it}{\sigma-1+it}\right|\left|G(\sigma+it)\right|/\log t
    \end{equation}
    is decreasing in $t$ for $1\leq\sigma\leq2$. Furthermore, notice that \eqref{weihnachts-quiz} is increasing in $\sigma$ for $1\leq\sigma\leq2$ and $t\geq653$. Therefore, we obtain 
    \begin{equation*}
        \left|\frac{\sigma+it}{\sigma-1+it}G(\sigma+it)\right|/\log t \leq C_1(t)
    \end{equation*}
    for $1\leq\sigma\leq1+\delta<2$, $t\geq653$ and $C_1(653)<1+6.23\cdot10^{-6}$. We thus obtain
    \begin{equation*}
        |\zeta(s)|\leq C_1(653)\zeta(1+\delta)^\frac{\sigma-1}{\delta}(0.548\log t)^\frac{1+\delta-\sigma}{\delta}.
    \end{equation*}
    To prove the second inequality we apply Theorem 1 from \cite{fiori} to the function $H(s)=\frac{s-1}{sG(s)}\zeta(s)$. For $0<\sigma<2$ and $t\geq653$ we have $\log t\leq |G(\sigma+it)|$ and together with \eqref{hiary,patel,yang} and \eqref{qingyi,teo} we obtain
    \begin{equation*}
        \left|H\left(\frac{1}{2}+it\right)\right|\leq0.618 |\sigma+it|^\frac{1}{6},\ \text{ and }|H(1+it)|\leq0.548.
    \end{equation*}
    The function $|\sigma+it|^\frac{1}{6}$ is bounded below and is increasing in $\sigma$. So we may apply Theorem 1 \cite{fiori} and obtain
    \begin{equation*}
        |H(\sigma+it)|\leq0.548^{2\sigma-1}\left(0.618|\sigma+it|^\frac{1}{6}\right)^{2-2\sigma}\leq0.548^{2\sigma-1}\left(0.618t^\frac{1}{6}\right)^{2-2\sigma}\tilde{C}(\sigma,t),
    \end{equation*}
    where $\tilde{C}(\sigma,t)=\left(1+\frac{\sigma^2}{t^2}\right)^\frac{1}{12}$ is decreasing in $t$ and
    and $\tilde{C}(\sigma,t)\leq1+1.96\cdot10^{-7}$ for $t\geq 653$ and $0\leq\sigma\leq1$. Thus,
    \begin{equation*}
        |\zeta(\sigma+it)|\leq0.548^{2\sigma-1}\left(0.618t^\frac{1}{6}\right)^{2-2\sigma}\left|\frac{\sigma+it}{\sigma-1+it}G(\sigma+it)\right|(1+1.96\cdot10^{-7}).
    \end{equation*}
    Lastly, we use that $\left|\frac{\sigma+it}{\sigma-1+it}G(\sigma+it)\right|\leq C_1(653)\log t$ for $\frac{1}{2}\leq\sigma\leq1$ and $t\geq653$. Hence,
    \begin{equation*}
        |\zeta(\sigma+it)|\leq C_2(653)0.548^{2\sigma-1}\left(0.618t^\frac{1}{6}\right)^{2-2\sigma}\log t
    \end{equation*}
    for $\frac{1}{2}\leq\sigma\leq1$ and $t\geq653$. Here, $C_2(653)\leq C_1(653)(1+1.96\cdot10^{-7})\leq1+6.43\cdot10^{-6}$.
\end{proof}

With the help of these results, we find the following estimate for \eqref{integral re log zeta}.
\begin{lemma}\label{upper bound}
    For $t\geq t_0\geq653$ and $r>1$, we have
    \begin{equation*}
        \re\int_{\frac{1}{2}+it}^{\infty+it}\log\zeta(s)ds\leq a_1+b_1\log\log t+c_1\log t,
    \end{equation*}
    where
    \begin{equation*}
        \begin{split}
            a_1=&\frac{1}{2}\log C_2(653)+\frac{\log(0.618)+\log(0.548)}{4}\\
            &+\frac{r-1}{2}\left(\log C_1(653)+\log(0.548)+\log\zeta(r)\right)+\int_{r}^\infty\log\zeta(\sigma)d\sigma,\\
            b_1=&\frac{r}{2},\\
            c_1=&\frac{1}{48}.
        \end{split}
    \end{equation*}
\end{lemma}
\begin{proof}
    Applying lemma \ref{fiori} yields
    \begin{equation}\label{integral 1/2, 1}
        \begin{split}
            \re\int_{\frac{1}{2}+it}^{1+it}\log\zeta(s)ds\leq&\int_{\frac{1}{2}}^1\log C_2(653)+2(1-\sigma)\log(0.618)+2\left(\sigma-\frac{1}{2}\right)\log(0.548)d\sigma\\
            &+\int_\frac{1}{2}^1\frac{1-\sigma}{3}\log t+\log\log t\ d\sigma\\
            \leq&\frac{1}{2}\log C_2(653)+\frac{\log(0.618)+\log(0.548)}{4}+\frac{1}{2}\log\log t+\frac{1}{48}\log t
        \end{split}
    \end{equation}
    for $t\geq653$. Next, we need to estimate 
    \begin{equation*}
        \re\int_{1+it}^{\infty+it}\log\zeta(s)ds.
    \end{equation*}
    First, we split the integral and then apply lemma \ref{fiori} as well as the trivial bound $|\zeta(s)|\leq\zeta(\sigma)$ to estimate the two parts. We obtain that
    \begin{equation*}
        \begin{split}
            \re\int_{1+it}^{\infty+it}\log\zeta(s)ds\leq&\int_{1+it}^{r+it}\log|\zeta(s)|ds+\int_r^\infty\log\zeta(\sigma)d\sigma,
        \end{split}
    \end{equation*}
    now apply lemma \ref{fiori} with $\delta=r-1$. This yields
    \begin{equation*}\label{r estimate integral}
        \begin{split}
            \re\int_{1+it}^{\infty+it}\log\zeta(s)ds\leq&\int_1^r\frac{r-\sigma}{r-1}\log(0.548\log t)+\frac{\sigma-1}{r-1}\log\zeta(r)-\log\zeta(\sigma)d\sigma\\
            &+(r-1)\log C_1(653)+\int_1^\infty\log\zeta(\sigma)d\sigma.
        \end{split}
    \end{equation*}
    Combining this estimate with \eqref{integral 1/2, 1} proves the lemma.
\end{proof}
The choice of parameters $t_0$ and $r$ will be made explicit in the next section. In order to find a lower bound of \eqref{integral re log zeta} we need the following lemma.

\begin{lemma}\label{key lemma}
    Let $\omega\in\C$ with $|\re(\omega)|\leq\frac{1}{2}$, then for all $d\in\left[0.606,1\right]$ we have that
    \begin{equation}\label{das ist big}
        \begin{split}
            \int_0^d\log\left|\frac{(x+d+\omega)(x+d-\bar{\omega})}{(x+\omega)(x-\bar{\omega})}\right|dx\leq&d^2q_1\re\left(\frac{1}{d+\omega}+\frac{1}{d-\bar{\omega}}\right)\\&+d^3q_2\re\left(\frac{1}{(d+\omega)^2}+\frac{1}{(d-\omega)^2}\right).
        \end{split}
    \end{equation}
    Where $q_1=\log(2)+\frac{1}{2}$ and $q_2=\log(2)-\frac{1}{2}$.
\end{lemma}
\begin{proof}
    We consider the function $H_d(\omega)$ given by
    \begin{equation*}
        \begin{split}
            H_d(\omega)=&\int_0^d\log\left|\frac{(x+d+\omega)(x+d-\bar{\omega})}{(x+\omega)(x-\bar{\omega})}\right|dx-d^2q_1\re\left(\frac{1}{d+\omega}+\frac{1}{d-\bar{\omega}}\right)\\
            &-d^3q_2\re\left(\frac{1}{(d+\omega)^2}+\frac{1}{(d-\omega)^2}\right).
        \end{split}
    \end{equation*}
    After a substitution, we obtain that
    \begin{equation*}
        H_d(\omega)=dH_1\left(\omega/d\right)
    \end{equation*}
    Notice that this is the real part of an analytic function. Furthermore, notice that $H_1(\sigma+it)\rightarrow0$. So, the function is bounded and we may apply the Phragmén-Lindelöf principle \cite[pp.180-181]{titchmarsh2} to verify that $H_1(\omega)\leq0$ in the area with $|\re(\omega)|\leq\frac{1}{2d}$.
    On the real line we can easily verify that
    \begin{equation*}
        \begin{split}
            \frac{d}{d\sigma}H_1(\sigma)=&\log\left(1-\frac{4\sigma}{(1-\sigma)^2(2+\sigma)}\right)-q_1\left(\frac{1}{(1-\sigma)^2}-\frac{1}{(1+\sigma)^2}\right)\\
            &-2q_2\left(\frac{1}{(1-\sigma)^3}-\frac{1}{(1+\sigma)^3}\right).
        \end{split}
    \end{equation*}
    Each of these terms is positive for negative $\sigma$ and negative for positive $\sigma$ for $\sigma\in(-1,1)$. Hence, $H_1(\sigma)$ is maximized at $0$. We deduce that
    \begin{equation*}
        H_d(\sigma)=dH_1(\sigma/d,q_1,q_2)\leq dH_1(0)=2d\log(4)-2dq_1-2dq_2=0
    \end{equation*}
    for $\sigma\in\left[-\frac{1}{2d},\frac{1}{2d}\right]\subset\left[-\frac{1}{2\cdot0.606},\frac{1}{2\cdot0.606}\right]$. Next, we need to verify that $H_1\left(\frac{1}{2d}+it\right)\leq0$ for all $t\in\R$ and all $d\in\left[0.606,1\right]$. It is sufficient to test one side of the boundary of $|\re(\omega)|\leq\frac{1}{2d}$ for this lemma due to the symmetry of $H_1$ about the line with real part $0$ and the Phragmén-Lindelöf principle. The line with real part $\frac{1}{2d}$ is contained in the area bounded by $|\re(s)|\leq\frac{1}{2\cdot0.606}$. Computing the derivative of $H_1\left(\frac{1}{2\cdot0.606}+it\right)$ with respect to $t$ returns
    \begin{equation*}
        \begin{split}
            \frac{d}{dt}H_1\left(\sigma^*+it\right)=&\arctan\left(\frac{\sigma^*+2}{t}\right)-2\arctan\left(\frac{\sigma^*+1}{t}\right)\\
            &+2\arctan\left(\frac{\sigma^*-1}{t}\right)-\arctan\left(\frac{\sigma^*-2}{t}\right)\\
            &+2tq_1\left(\frac{1+\sigma^*}{((1+\sigma^*)^2+t^2)^2}+\frac{1-\sigma^*}{((1-\sigma^*)^2+t^2)^2}\right)\\
            &+2tq_2\left(\frac{3(1+\sigma^*)^2-t^2}{((1+\sigma^*)^2+t^2)^3}+\frac{3(1-\sigma^*)^2-t^2}{((1-\sigma^*)^2+t^2)^3}\right),
        \end{split}
    \end{equation*}
    where $\sigma^*=\frac{1}{2\cdot0.606}$. For $t\geq10$ we may compute the $\arctan$-terms by their taylor series. Similarly, we may rewrite the contribution of the poles as geometric series. Overall, we obtain that we may express $\frac{d}{dt}H_1(\sigma^*+it)$ by
    \begin{equation*}
        \begin{split}
            \frac{d}{dt}H_1(\sigma^*+it)=&\sum_{k=0}^\infty(-1)^k\frac{(\sigma^*+2)^{2k+1}+(\sigma^*-2)^{2k+1}}{2k+1}t^{-2k-1}\\
            &-2\sum_{k=0}^\infty(-1)^k\frac{(\sigma^*+1)^{2k+1}+(\sigma^*-1)^{2k+1}}{2k+1}t^{-2k-1}\\
            &-q_1\sum_{k=1}^\infty(-1)^k2k\left((1+\sigma^*)^{2k-1}+(1-\sigma^*)^{2k-1}\right)t^{-2k-1}\\
            &+q_2\sum_{k=1}^\infty(-1)^k2k(2k-1)\left((1+\sigma^*)^{2k-2}+(1-\sigma^*)^{2k-2}\right)t^{-2k-1}.
        \end{split}
    \end{equation*}
    Notice that the the terms up to order $t^{-3}$ cancel and that we may lower bound the remaining terms by the alternating series criterion. Here we use that $t\geq10$ is much larger than $|\sigma^*\pm1|$ and $|\sigma^*\pm2|$. We obtain that
    \begin{equation*}
        \frac{d}{dt}H_1(\sigma^*+it)\geq\frac{16\log(2)-4}{t^5}-\frac{270}{t^7}>0
    \end{equation*}
    for $t\geq10$ and deduce that all local extrema of $H_1(\sigma^*+it)$ must lie in the range with $|t|<10$. By a numerical analysis we obtain that $H_1(\sigma^*+it)$ has local extrema at $t\approx\pm0.738$ and at $t\approx\pm1.505$. The former correspond to local maxima and the latter correspond to local minima of $H_1(\sigma^*+it)$. Evaluating $H_1(\sigma^*+it)$ in the vicinity of $\pm0.738$ shows that $H_1(s)$ is not positive on the line with $\re(s)=\frac{1}{2\cdot0.606}$. And by the Phrangmén-Lindelöf principle $H_1(\frac{1}{2d}+it)\leq0$ for all $d\in\left[0.606,1\right]$.
\end{proof}
With this result at hand we may show the following lower bound for \eqref{integral re log zeta}.
\begin{lemma}\label{lemma lower bound}
    Let $t\geq653$ and $d\in[0.606,1]$. Then,
    \begin{equation*}
        -\re\int_{\frac{1}{2}+it}^{\infty+it}\log\zeta(s)ds\leq a_2+c_2\log t,
    \end{equation*}
    where
    \begin{equation*}
        \begin{split}
            a_2=&I(d)+\frac{d^2(1-q_1)}{2}\log(2\pi)+d^2q_1\epsilon-d^2q_1\left(\frac{\zeta'}{\zeta}\left(\frac{1}{2}+d\right)-2\frac{\zeta'}{\zeta}(1+2d)\right)\\
            &+d^3q_2\left(\left(\frac{\zeta''}{\zeta}-\left(\frac{\zeta'}{\zeta}\right)^2\right)\left(\frac{1}{2}+d\right)-4\left(\frac{\zeta''}{\zeta}-\left(\frac{\zeta'}{\zeta}\right)^2\right)(1+2d)+\frac{\log^2(2)}{8}+\frac{\log^2(3)}{128}\right)\\
            c_2=&\frac{d^2(q_1-1)}{2}
        \end{split}
    \end{equation*}
    with
    \begin{equation}\label{400000 jahre geschichte}
        \begin{split}
            I(d)=\int_{\frac{1}{2}+d}^\infty\log\zeta(\sigma)d\sigma+\int_{\frac{1}{2}+d}^{\frac{1}{2}+2d}\log\zeta(\sigma)d\sigma-\frac{1}{2}\int_{1+2d}^\infty\log\zeta(\sigma)d\sigma-\frac{1}{2}\int_{1+2d}^{1+4d}\log\zeta(\sigma)d\sigma,
        \end{split}
    \end{equation}
    and $\epsilon<3\cdot10^{-6}$. Here and throughout the proof $q_1=\log(2)+\frac{1}{2}$ and $q_2=\log(2)-\frac{1}{2}$.
\end{lemma}
\begin{proof}
    Write
    \begin{equation*}
        \begin{split}
            -\re\int_{\frac{1}{2}+it}^{\infty+it}\log\zeta(s)ds=&\re\int_{\frac{1}{2}+it}^{\frac{1}{2}+d+it}\log\frac{\zeta(s+d)}{\zeta(s)}ds\\
            &-\re\int_{\frac{1}{2}+d+it}^{\infty+it}\log\zeta(s)ds-\re\int_{\frac{1}{2}+d+it}^{\frac{1}{2}+2d+it}\log\zeta(s)ds\\
            \leq&\re\int_{\frac{1}{2}+it}^{\frac{1}{2}+d+it}\log\frac{\zeta(s+d)}{\zeta(s)}ds+I(d)
        \end{split}
    \end{equation*}
    where
    \begin{equation*}
        \begin{split}
            I(d)=\int_{\frac{1}{2}+d}^\infty\log\zeta(\sigma)d\sigma+\int_{\frac{1}{2}+d}^{\frac{1}{2}+2d}\log\zeta(\sigma)d\sigma-\frac{1}{2}\int_{1+2d}^\infty\log\zeta(\sigma)d\sigma-\frac{1}{2}\int_{1+2d}^{1+4d}\log\zeta(\sigma)d\sigma.
        \end{split}
    \end{equation*}
    Here, we used the inequality that $\left|\frac{1}{\zeta(s)}\right|\leq\frac{\zeta(\sigma)}{\zeta(2\sigma)}$ for $\sigma>1$. All of the integrals in equation \eqref{400000 jahre geschichte} are convergent and will be evaluated at the end of the proof. Next, recall the Weierstrass product formula \cite[pp. 82-83]{davenport} for $\zeta(s)$,
    \begin{equation*}
        \zeta(s)=\frac{e^{bs}}{2(s-1)\Gamma(1+\frac{s}{2})}\prod_\rho\left(1-\frac{s}{\rho}\right)e^{s/\rho},
    \end{equation*}
    where the product is taken over the non-trivial zeros of $\zeta(s)$ and the constant $b$ satisfies
    \begin{equation*}
        b=\frac{1}{2}\log\pi-\re\sum_\rho\frac{1}{\rho}.
    \end{equation*}
    This expression is absolutely convergent. Thus, we may use this product to rewrite $\log\zeta(s)$ as an infinite series involving the non-trivial zeros of $\zeta(s)$ and after grouping off critical line zeros with their symmetric counterpart ($\rho$ and $1-\bar{\rho}$), we obtain
    \begin{equation*}
        \begin{split}
            -\re\int_{\frac{1}{2}+it}^{\infty+it}\log\zeta(s)ds\leq&\frac{1}{2}\sum_\rho\int_{\frac{1}{2}+it}^{\frac{1}{2}+d+it}\log\left|\frac{(s-\rho+d)(s-(1-\bar{\rho})+d)}{(s-\rho)(s-(1-\bar{\rho}))}\right|ds\\
            &-\int_{\frac{1}{2}+it}^{\frac{1}{2}+d+it}\log\left|\frac{\Gamma\left(\frac{s+d}{2}+1\right)}{\Gamma\left(\frac{s}{2}+1\right)}\right|ds\\
            &+\int_{\frac{1}{2}+it}^{\frac{1}{2}+d+it}\log\left|\frac{s-1}{s+d-1}\right|ds-I(d)+\frac{d^2}{2}\log\pi\\
            =&I_1-I_2+I_3+I(d)+\frac{d^2}{2}\log\pi.
        \end{split}
    \end{equation*}
    The integral $I_3\leq0$ and approaches $0$ as $t\rightarrow\infty$. We can bound the contribution of $I_1$ with Lemma 4 and obtain that
    \begin{equation*}
        \begin{split}
            I_1\leq& d^2q_1\re\sum_\rho\frac{1}{\frac{1}{2}+d+it-\rho}+d^3\left(\log(2)-\frac{1}{2}\right)\re\sum_\rho\frac{1}{(\frac{1}{2}+d+it-\rho)^2}\\
            =&d^2q_1\re\frac{\xi'}{\xi}\left(\frac{1}{2}+d+it\right)-d^3q_2\re\left(\frac{\xi''}{\xi}-\left(\frac{\xi'}{\xi}\right)^2\right)\left(\frac{1}{2}+d+it\right)\\
            =&d^2q_1\left[\re\frac{\zeta'}{\zeta}\left(\frac{1}{2}+d+it\right)+\frac{d-\frac{1}{2}}{\left(d-\frac{1}{2}\right)^2+t^2}+\frac{1}{2}\re\psi\left(\frac{3}{4}+\frac{d}{2}+\frac{it}{2}\right)-\frac{\log\pi}{2}\right]\\
            &-d^3q_2\left(\re\left(\frac{\zeta''}{\zeta}-\left(\frac{\zeta'}{\zeta}\right)^2\right)\left(\frac{1}{2}+d+it\right)-\frac{\left(d-\frac{1}{2}\right)^2-t^2}{\left(\left(d-\frac{1}{2}\right)+t^2\right)^2}+\frac{1}{4}\psi_1\left(\frac{3}{4}+\frac{d}{2}+\frac{it}{2}\right)\right),
        \end{split}
    \end{equation*}
    where $\psi$ and $\psi_1$ denote the di- and trigamma function respectively. Notice that the contributions of the pole, $\frac{d-\frac{1}{2}}{\left(d-\frac{1}{2}\right)^2+t^2}$ and $-\frac{\left(d-\frac{1}{2}\right)^2-t^2}{\left(\left(d-\frac{1}{2}\right)^2+t^2\right)^2}$, are decreasing with $|t|$ as $|t|$ grows large. Furthermore, we use the estimates in \cite[cor. 2.5]{amberger2025estimatingnumberzerosdedekind} to bound the contributions of the di-and trigamma terms yielding that
    \begin{equation*}
        \begin{split}
            \frac{d-\frac{1}{2}}{\left(d-\frac{1}{2}\right)^2+t^2}+\frac{1}{2}\re\psi\left(\frac{3}{4}+\frac{d}{2}+\frac{it}{2}\right)\leq&\frac{1}{2}\log\frac{t}{2}+\epsilon\text{ and}\\
            \frac{\left(d-\frac{1}{2}\right)^2-t^2}{\left(\left(d-\frac{1}{2}\right)^2+t^2\right)}-\frac{1}{4}\re\psi_1\left(\frac{3}{4}+\frac{d}{2}+\frac{it}{2}\right)\leq&0,
        \end{split}
    \end{equation*}
    for $t\geq653$ where $\epsilon=3\cdot10^{-6}$. Furthermore, we need some estimates for the first and second logarithmic derivatives of the zeta function. Because $d>\frac{1}{2}$ we may use the Euler-product to represent the zeta function in this area. This yields the following representations for the logarithmic derivatives of zeta. We have
    \begin{equation*}
        \begin{split}
            \re\frac{\zeta'}{\zeta}(s)=&-\sum_{p\in\mathbb{P}}\frac{\log(p)p^\sigma\cos(t\log p)}{1-2\cos(t\log p)p^\sigma+p^{2\sigma}}\text{ and}\\
            -\re\left(\frac{\zeta''}{\zeta}-\left(\frac{\zeta'}{\zeta}\right)^2\right)(s)=&-\sum_{p\in\mathbb{P}}\frac{\log^2(p)((p^{3\sigma}+p^\sigma)\cos(t\log p)-2p^{2\sigma})}{(1-2\cos(t\log p)p^\sigma+p^{2\sigma})^2}.
        \end{split}
    \end{equation*}
    The function $-\frac{y}{1-2yx+x^2}$ is decreasing in $y$ for all $x\geq2$. Furthermore, consider the function given by
    \begin{equation*}
        f(x,y)=-\frac{x(x^2+1)y-2x^2}{(1-2xy+x^2)^2}-\frac{x}{(1+x)^2}
    \end{equation*}
    with partial derivatives
    \begin{equation*}
        \begin{split}
            \partial_x&f(x,y)=\\
            &\frac{(y+1)(x-1)(x^6+1+(x^4+1)(2y-4)x+(x^2+1)(20y-13)x^2-(8y^2-20y+24)x^3)}{(x+1)^3(x^2-2xy+1)^3},\\
            \partial_y&f(x,y)=-\frac{x(x^4+2x^3y-6x^2+2xy+1)}{(x^2-2xy+1)^3}.
        \end{split}
    \end{equation*}
    We observe that for $x\geq2$ the local extrema of $f(x,\cdot)$ are achieved if and only if $x^4+2x^3y-6x^2+2xy+1=0$. For $x\geq2$ this has a solution with $y\in[-1,1]$ only when $x\in[2,2+\sqrt{3}]$. Solving for $y$ and substituting in $\partial_xf(x,y)$ yields 
    \begin{equation*}
        \partial_x\left( \max_{y\in[-1,1]}f(x,y)\right)=\frac{x^4-5x^3+6x^2-5x+1}{(x^2-1)^3}\leq0
    \end{equation*} for $x\in[2,2+\sqrt{3}]$. Hence, $g(x):=\max_{y\in(-1,1]}f(x,y)$ is decreasing for $x\geq2$ with $g(2)=\frac{1}{8}$, $g(3)=\frac{1}{128}$ and $g(x)=0$ for $x\geq2+\sqrt{3}$. We thus obtain the
    \begin{equation}\label{backzeit 12 min}
        \begin{split}
            \re\frac{\zeta'}{\zeta}(s)\leq&-\frac{\zeta'}{\zeta}(\sigma)+2\frac{\zeta'}{\zeta}(2\sigma)\text{ and}\\
            -\re\left(\frac{\zeta''}{\zeta}-\left(\frac{\zeta'}{\zeta}\right)^2\right)(s)\leq&\left(\frac{\zeta''}{\zeta}-\left(\frac{\zeta'}{\zeta}\right)^2\right)(\sigma)-4\left(\frac{\zeta''}{\zeta}-\left(\frac{\zeta'}{\zeta}\right)^2\right)(2\sigma)\\
            &+\frac{\log^2(2)}{8}+\frac{\log^2(3)}{128}.
        \end{split}
    \end{equation}
    Therefore,
    \begin{equation*}
        \begin{split}
            I_1\leq&\frac{d^2}{2}q_1\log\frac{t}{2\pi}+d^2q_1\epsilon+d^2\max_{t\in\R}\left\{\left[q_1\re\frac{\zeta'}{\zeta}-dq_2\re\left(\frac{\zeta''}{\zeta}-\left(\frac{\zeta'}{\zeta}\right)^2\right)\right]\left(\frac{1}{2}+d+it\right)\right\}\\
            \leq&\frac{d^2}{2}q_1\log\frac{t}{2\pi}+d^2q_1\epsilon-d^2q_1\left(\frac{\zeta'}{\zeta}\left(\frac{1}{2}+d\right)-2\frac{\zeta'}{\zeta}(1+2d)\right)\\
            +&d^3q_2\left(\left(\frac{\zeta''}{\zeta}-\left(\frac{\zeta'}{\zeta}\right)^2\right)\left(\frac{1}{2}+d\right)-4\left(\frac{\zeta''}{\zeta}-\left(\frac{\zeta'}{\zeta}\right)^2\right)(1+2d)+\frac{\log^2(2)}{8}+\frac{\log^2(3)}{128}\right).
        \end{split}
    \end{equation*}
    To evaluate $I_2$ we apply the mean value theorem for integrals. Thus, we see that
    \begin{equation*}
        \begin{split}
            -I_2=&-\frac{1}{2}\int_{\frac{1}{2}+it}^{\frac{1}{2}+d+it}\int_0^d\re\psi\left(\frac{s+z}{2}+1\right)dzds=-\frac{d^2}{2}\re\psi\left(\sigma_0+\frac{it}{2}\right)
        \end{split}
    \end{equation*}
    for some $\sigma_0\in\left(\frac{5}{4},\frac{5}{4}+d\right)$. We may now bound this contribution by referring to the digamma estimates in \cite[cor. 2.5]{amberger2025estimatingnumberzerosdedekind}. We have that
    \begin{equation*}
        -I_2\leq-\frac{d^2}{2}\log\frac{t}{2}.
    \end{equation*}
    By combining these bounds we obtain that
    \begin{equation*}\label{lower bound a2,c2}
        -\re\int_{\frac{1}{2}+it}^{\infty+it}\log\zeta(s)ds\leq a_2+c_2\log t
    \end{equation*}
    with
    \begin{equation*}
        \begin{split}
            a_2=&I(d)+\frac{d^2(1-q_1)}{2}\log(2\pi)+d^2q_1\epsilon-d^2q_1\left(\frac{\zeta'}{\zeta}\left(\frac{1}{2}+d\right)-2\frac{\zeta'}{\zeta}(1+2d)\right)\\
            &+d^3q_2\left(\left(\frac{\zeta''}{\zeta}-\left(\frac{\zeta'}{\zeta}\right)^2\right)\left(\frac{1}{2}+d\right)-4\left(\frac{\zeta''}{\zeta}-\left(\frac{\zeta'}{\zeta}\right)^2\right)(1+2d)+\frac{\log^2(2)}{8}+\frac{\log^2(3)}{128}\right)\\
            c_2=&\frac{d^2(q_1-1)}{2},
        \end{split}
    \end{equation*}
    thereby proving the lemma.
\end{proof}
\begin{remark}\label{tatort time}
    Here we got another small improvement upon Trudgian's bound through the use of a sharper lower bound for the real part of the logarithmic derivative. The idea is similar to the trivial lower bound for the zeta function that states that $|\zeta(s)|\geq\frac{\zeta(2\sigma)}{\zeta(\sigma)}$. This bound arises from considering the euler-product for $\zeta(s)$,
    \begin{equation*}
        \sum_{k=0}^\infty p^{-ks}=\sum_{k=1}^\infty p^{-k\sigma}e^{-i\log(p)tk}.
    \end{equation*}
    Notice that $\re(e^{-i\log(p)2kt})$ is positive if $\re(e^{-i\log(p)kt})$ is negative. Hence, cancellations arise between terms associated with even and odd prime powers, which, in turn, suggests the presence of the $\zeta(2\sigma)$ factor. The same cancellations between odd and even prime powers give rise to the improved estimates of the first and second logarithmic derivative of $\zeta(s)$ in \eqref{backzeit 12 min}.
\end{remark}
\section{Computations and comparison to previous estimates}\label{section Tables}
In this section, we compute admissible values for $(a_1,b_1,c_1)$ and $(a_2,c_2)$ and prove theorem \ref{S_1 estimates theorem}. First, we search for admissible $(a_1,b_1,c_1)$. We wish to optimize the upper bound in lemma \ref{upper bound} at a specific height $t'\geq t_0\geq653$. We do so by varying $r\in\{1.002,1.004,\dots,1.5\}$ and choosing a minimizer from this subset of $(1,\infty)$. We present our findings in table \eqref{table: a} on the next page.  We only list the parameters $a_1$ and $b_1$ because $c_1=\frac{1}{48}$ was chose independent of $r$. Furthermore, we want to optimize $a_2$ and $c_2$ at fixed heights $t'$. To optimize the bound with respect to $d$ we resort to testing $d\in\{0.606,0.607,\dots,1\}$. Table \eqref{table: d} contains the results for $(d,a_2,b_2)$.
\begin{table}[h]
\parbox{.45\linewidth}{
\centering
\begin{tabular}{||c c c c||}
    \hline
    $t'$ & $r$ &$a_1$ & $b_1$ \\ [0.5ex] 
    \hline\hline
    $10^3$ & $1.482$ & $1.132$ & $0.741$ \\ 
    $10^4$ & $1.404$ & $1.213$ & $0.702$ \\
    $10^5$ & $1.352$ & $1.274$ & $0.676$\\
    $10^6$ & $1.316$ & $1.319$ & $0.658$\\
    $10^7$ & $1.290$ & $1.354$ & $0.645$ \\
    $10^8$ & $1.268$ & $1.386$ & $0.634$ \\ 
    $10^9$ & $1.252$ & $1.409$ & $0.626$ \\
    $10^{10}$ & $1.236$ & $1.434$ & $0.618$ \\
    $10^{11}$ & $1.224$ & $1.453$ & $0.612$ \\
    $10^{12}$ & $1.214$ & $1.470$ & $0.607$ \\
    $10^{13}$ & $1.204$ & $1.486$ & $0.602$ \\
    $10^{14}$ & $1.196$ & $1.500$ & $0.598$ \\
    $10^{15}$ & $1.188$ & $1.514$ & $0.594$ \\
    $10^{16}$ & $1.182$ & $1.525$ & $0.591$ \\
    $10^{17}$ & $1.176$ & $1.536$ & $0.588$ \\
    $10^{18}$ & $1.17$ & $1.547$ & $0.585$ \\[1ex]  
    \hline
    \end{tabular}
    \caption{admissible values $(t',a_1,b_1)$ for $t_0=653$}
    \label{table: a}
}
\hfill
\parbox{.45\linewidth}{
\centering
    \begin{tabular}{||c c c c||}
    \hline
    $t'$ & $d$ & $a_2$ & $c_2$ \\ [0.5ex] 
    \hline\hline
    $10^3$ & $1.000$ & $3.068$ & $0.0966$ \\ 
    $10^7$ & $1.000$ & $3.068$ & $0.0966$ \\
    $10^8$ & $0.977$ & $3.145$ & $0.0922$ \\
    $10^9$ & $0.952$ & $3.236$ & $0.0876$ \\ 
    $10^{10}$ & $0.931$ & $3.320$ & $0.0838$\\
    $10^{11}$ & $0.913$ & $3.397$ & $0.0806$\\
    $10^{12}$ & $0.897$ & $3.471$ & $0.0778$\\
    $10^{13}$ & $0.883$ & $3.540$ & $0.0753$\\
    $10^{14}$ & $0.871$ & $3.603$ & $0.0733$\\
    $10^{15}$ & $0.859$ & $3.670$ & $0.0713$\\
    $10^{16}$ & $0.849$ & $3.729$ & $0.0697$\\
    $10^{17}$ & $0.840$ & $3.785$ & $0.0682$\\
    $10^{18}$ & $0.832$ & $3.837$ & $0.0669$\\[1ex] 
    \hline
    \end{tabular}
    \caption{admissible values for $(t',a_2,c_2)$}
    \label{table: d}    }
\end{table}

With these values and lemmas \ref{switch from vertical to horizontal lines}, \ref{upper bound}, and \ref{lemma lower bound} we conclude that the following table contains admissible choices for $(a,b,c)$ such that \ref{S_1 estimates theorem} holds true. Table \eqref{table: e} 
contains admissible choices for $(a,b,c)$.
\newpage
\begin{table}
\centering
\begin{tabular}{||c c c c||}
    \hline
    $t'$ & $a$ &$b$ & $c$ \\ [0.5ex] 
    \hline\hline
    $10^3$ & $1.337$ & $0.236$ & $0.0374$ \\ 
    $10^4$ & $1.363$ & $0.224$ & $0.0374$ \\
    $10^5$ & $1.383$ & $0.216$ & $0.0374$\\
    $10^6$ & $1.397$ & $0.210$ & $0.0374$\\
    $10^7$ & $1.408$ & $0.206$ & $0.0374$ \\
    $10^8$ & $1.443$ & $0.202$ & $0.0360$ \\ 
    $10^9$ & $1.479$ & $0.200$ & $0.0345$ \\
    $10^{10}$ & $1.514$ & $0.197$ & $0.0333$ \\
    $10^{11}$ & $1.544$ & $0.195$ & $0.0323$ \\
    $10^{12}$ & $1.573$ & $0.194$ & $0.0314$ \\
    $10^{13}$ & $1.600$ & $0.192$ & $0.0306$ \\
    $10^{14}$ & $1.625$ & $0.191$ & $0.0300$ \\
    $10^{15}$ & $1.651$ & $0.190$ & $0.0294$ \\
    $10^{16}$ & $1.673$ & $0.189$ & $0.0288$ \\
    $10^{17}$ & $1.694$ & $0.188$ & $0.0284$ \\
    $10^{18}$ & $1.714$ & $0.187$ & $0.0280$ \\[1ex]  
    \hline
    \end{tabular}
    \caption{admissible values $(t',a,b,c)$ for $t_0=653$}
    \label{table: e}
\end{table}
\begin{remark}
    By choosing $r=1$ and $d=0.606$ one may optimize the estimate in the $\log t$ and $\log\log t$ aspect. This way, the constant $c$ may be taken as small as $0.018$, yielding
    \begin{equation*}
        |S_1(t_2)-S_1(t_1)|\leq3.355+0.160\log\log t_2 +0.018\log t_2
    \end{equation*}
    for $653\leq t_1<t_2$.
\end{remark}
Here is a comparison of the bounds obtained from this method compared to Trudgian's bounds for $|S_1(t_2)-S_1(t_1)|$.
    \begin{table}[h!]
    \centering
    \begin{tabular}{||c c c c c c||}
    \hline
    $t'$ & Table 1 \cite{Trudgian2} & Table \eqref{table: e} & $a$ & $b$ & $c$ \\ [0.5ex] 
    \hline\hline
    $10^{5}$ & $2.629$ & $2.342$ & $1.383$ & $0.216$ & $0.0374$ \\
    $10^{6}$ & $2.800$ & $2.466$ & $1.397$ & $0.210$ & $0.0374$ \\
    $10^{7}$ & $2.959$ & $2.584$ & $1.408$ & $0.206$ & $0.0374$ \\
    $10^{8}$ & $3.110$ & $2.695$ & $1.443$ & $0.202$ & $0.0360$ \\
    $10^{9}$ & $3.255$ & $2.801$ & $1.479$ & $0.200$ & $0.0345$ \\
    $10^{10}$ & $3.395$ & $2.899$ & $1.514$ & $0.197$ & $0.0333$ \\
    $10^{11}$ & $3.526$ & $2.993$ & $1.544$ & $0.195$ & $0.0323$ \\
    $10^{12}$ & $3.649$ & $3.085$ & $1.573$ & $0.194$ & $0.0314$ \\
    $10^{13}$ & $3.770$ & $3.169$ & $1.600$ & $0.192$ & $0.0306$ \\
    $10^{14}$ & $3.887$ & $3.262$ & $1.625$ & $0.191$ & $0.0302$ \\
    $10^{15}$ & $4.002$ & $3.340$ & $1.651$ & $0.190$ & $0.0294$ \\[1ex]  
    \hline
    \end{tabular}
    \label{table: h}
    \end{table}

Lastly, we use Theorem \ref{theorem abc loglog t} to find estimates for $K$ at specific heights $t$. Using the estimates in table \eqref{table: e} 
one may find that one needs $9$ consecutive Gram blocks that satisfy Rosser's rule to apply theorem \ref{theorem abc loglog t} at height $\leq3.6\cdot10^{12}$, at height $\leq3.6\cdot10^{11}$, $8$ such Gram blocks are needed and at heights $\leq3.2\cdot10^{10}$, we require $7$ consecutive Gram blocks to apply the theorem.

\section{Conclusion}
Notice that the $\log t$ component of the bounds which we obtained in lemma \ref{lemma lower bound} are linearly dependent on the coefficient $q_1$ from lemma \ref{key lemma}. This lemma marks an improvement over previous versions of this lemma (see lemma 2.10 in \cite{Trudgian1}) stemming from our consideration of second order poles in addition to simple poles to estimate the integral expression on the right-hand side of equation \eqref{das ist big}. One may try to improve the constant further through consideration of more higher order poles. Furthermore, it should be noted that in order to apply lemma \ref{key lemma} to estimate logarithmic integrals of $L$-functions, we merely require that all non-trivial zeros lie inside the critical strip. Hence, the method is applicable to the study of the distribution of zeros of other $L$-functions.




\vspace{1em}
A documentation of the python code can be found at \url{https://arxiv.org/abs/2512.23064}
\begin{center}
    {\Large {Acknowledgment}}\\
I thank Valentin Blomer for his feedback on the paper and I thank my friend Hendrik Ehrhardt who helped me set up the Python code for the numerical computations. Furthermore, I thank Andrew Fiori for pointing out a mistake in a previous version of this paper.
\end{center}

\printbibliography

  \footnotesize

  \textsc{Mathematisches Institut, Endenicher Allee 60, 53115 Bonn, Germany}\par\nopagebreak
  \textit{E-mail address}: \texttt{victor.amberger@gmx.de}

\end{document}